\documentclass[12pt,reqno]{amsart}
\usepackage{amsmath, amsthm, amscd, amsfonts, amssymb, graphicx}
\setbox0=\hbox{$+$}
\newdimen\plusheight
\plusheight=\ht0
\def\+{\;\lower\plusheight\hbox{$+$}\;}

\setbox0=\hbox{$-$}
\newdimen\minusheight
\minusheight=\ht0
\def\-{\;\lower\minusheight\hbox{$-$}\;}

\setbox0=\hbox{$\cdots$}
\newdimen\cdotsheight
\cdotsheight=\plusheight
\def\cds{\lower\cdotsheight\hbox{$\cdots$}}

\renewcommand{\(}{\left\(}
\renewcommand{\)}{\right\)}

\renewcommand{\pmod}[1]{\,(\textup{mod}\,#1)}
\numberwithin{equation}{section}
\theoremstyle{plain}
\newtheorem{theorem}{Theorem}[section]
\newtheorem{lemma}[theorem]{Lemma}

\begin{document}
\noindent{\bf Infinite Families of Congruences Modulo 5 for Ramanujan's General Partition Function}\vskip3mm

\centerline{\bf Nipen Saikia\footnotemark[1] \quad and\quad Jubaraj Chetry}
\footnotetext[1]{\Small{\textit{Corresponding author.}}}
\begin{center}\textit{Department of Mathematics, Rajiv Gandhi University, \\Rono Hills, Doimukh-791112, Arunachal Pradesh, India.}\\
\textit{E. Mail: nipennak@yahoo.com (N. Saikia)\\\hskip3cm jubarajchetry470@gmail.com(J. Chetry)}\end{center}\vskip 3mm

\noindent{\footnotesize{\bf Abstract:}} For any non-negative integer $n$ and non-zero integer  $r$, let $p_r(n)$  denote Ramanujan's general partition function. By employing $q$-identities, we prove some new Ramanujan-type congruences modulo 5 for  $p_r(n)$ for $r=-(5\lambda+1), -(5\lambda+3), -(5\lambda+4), -(25\lambda+1)$, $-(25\lambda+2)$, and any integer $\lambda$.
\vskip 5mm

\vskip 2mm

 
\section{Introduction}A partition of a positive integer $n$ is a non-increasing sequence of positive integers, called parts, whose sum equals $n$. For example, $n=3$ has three partitions, namely, 
  	$$3,\quad 2+1,\quad 1+1+1.$$ If $p(n)$ denote the number of partitions of $n$, then $p(3)=3$. The generating function for  $p(n)$ is given by
  	  	\begin{equation}\label{pngen}\sum_{n=0}^\infty p(n)q^n=\dfrac{1}{(q; q)_\infty}, \quad p(0)=1,%
  	  	\end{equation}where, here and throughout this paper \begin{equation}(a;q)_\infty=\prod_{k=1}^\infty(1-aq^{k-1}).\end{equation}  	Ramanujan \cite{sr, sr2} proved following  beautiful congruences for $p(n)$:
  	  	\begin{equation}p(5n+4)\equiv0 \pmod 5,
  	  	\end{equation}
  	  	\begin{equation}p(7n+5)\equiv0 \pmod 7,
  	  	\end{equation} and
  	  	\begin{equation}p(13n+6)\equiv0 \pmod {11}.
  	  	\end{equation} 
 In a letter to Hardy written from Fitzroy House late in 1918 \cite[p. 192-193]{bc1}, Ramanujan introduced the general partition function $p_r(n)$ for integers $n\ge 0$ and  $r\ne 0$ as \begin{equation}\label{eq1}
{\sum_{n=0}^\infty{p_r(n)q^n}=\frac{1}{(q;q)_\infty^r}}, \qquad |q|<1.
\end{equation}
In order to describe the partition function $p_r(n)$, we first give the notion of colour partition.  A part in a partition of $n$ is said to have  $r$ colours if there are $r$ copies of each part available and all of them are viewed as distinct objects.

Now, for $r>1$, $p_r(n)$ denotes the number of partitions of $n$ where each part may have $r$ distinct colours.  For example,  if each part in the partitions of 3 have \textit{two} colours, say red and green, then the number of two colour partitions of $3$ is 10, namely  $$3_r,\quad 3_g, \quad2_r+1_r, \quad2_r+1_g,\quad 2_g+1_g, \quad 2_g+1_r,$$$$ 1_r+1_r+1_r,\quad 1_g+1_g+1_g,\quad 1_r+1_g+1_g,\quad 1_r+1_r+1_g.$$ Thus, $p_2(3)=10$. For $r=1$, $p_1(n)$ is the usual partition function $p(n)$ which counts the number of unrestricted partitions of a non-negative integer $n$.  Gandhi \cite{JM} studied the colour partition function $p_r(n)$ and found Ramanujan-type congruences for  certain values of $r$.  For example, he proved that 
$$p_2(5n+3)\equiv 0\pmod 5 $$and
$$p_8(11n+4)\equiv 0\pmod{11}.$$ Newman \cite{ab} also found some congruences for colour partition. Recently, Hirschhorm \cite{Hirs} found congruences for $p_3(n)$ modulo higher powers of 3.
 
If $r<0$, then \begin{equation}\label{mr}p_r(n)=\left(p_r(n, e)-p_r(n, o)\right),\end{equation} where $p_r(n, e)$ (resp. $p_r(n, o)$) is the number of partitions of $n$ with even (resp. odd) number of distinct parts  and each part have $r$ colours. For example, if $n=5$ and $r=-1$ then $p_{-1}(5, e)=2$ with relevant partitions $4+1$ and $3+2$, and $p_{-1}(5, o)=1$ with the relevant partition 5. Thus, $p_{-1}(5)=2-1=1$. Similarly, we see that $p_{-2}(3)=4-2=2$. The case $r=-1$ in \eqref{mr} is the famous Euler's pentagonal number theorem.

Ramanujan \cite{bc1} showed that, if $\lambda$ is a positive integer and  $\overline{w}$ is a prime of the form $6\lambda-1$, then  
\begin{equation}
p_{-4}\left(n\overline{w}-\frac{(\overline{w}+1)}{6}\right)\equiv0\pmod{\overline{w}}.
\end{equation} 
Ramanathan \cite{kg}, Atkin \cite{ao}, and Ono \cite{ko}. 
Baruah and Ojah \cite{NK} also proved some congruences for $p_{-3}(n)$. Recently, Baruah and Sharma \cite{NB} proved some arithmetic identities and congruences of  $p_r(n)$ for some particular negative values of $r$. 
    
In this paper, we prove some new Ramanujan-type congruences modulo 5 for  $p_r(n)$ for  $r=-(5\lambda+1), -(5\lambda+3), -(5\lambda+4), -(25\lambda+1)$, and $-(25\lambda+2)$, where  $\lambda$ is any integer, by employing some $q$-identities. In particular, we prove the following infinite families of congruences modulo 5 for the $p_r(n)$:

 \begin{theorem}\label{thm0} For any integer $\lambda$ and $\ell=3, 4$, we have
	$$p_{-(5\lambda+1)}(5n+\ell)\equiv0\pmod5.$$
\end{theorem}

 \begin{theorem}\label{thm1} For any integer $\lambda$ and $\ell=2, 3, 4$, we have
$$p_{-(5\lambda+3)}(5n+\ell)\equiv0\pmod5.$$
\end{theorem}

\begin{theorem}\label{thm3} For any  integer $\lambda$, we have
$$p_{-(5\lambda+4)}(5n+4)\equiv0\pmod5.$$
\end{theorem}

\begin{theorem}\label{thm5} For any integer $\lambda$ and $\ell=1,2,3,4$, we have
	$$p_{-(25\lambda+1)}(25n+5\ell+1)\equiv0\pmod5.$$
\end{theorem}

\begin{theorem}\label{thm6} For any  integer $\lambda$ and $\ell=1,2,3,4$, we have
	$$p_{-(25\lambda+2)}(25n+5\ell+2)\equiv0\pmod5.$$
\end{theorem}

\section{Preliminaries}
To prove Theorems \ref{thm0}-\ref{thm6} we will employ some $q-$identities.
Ramanujan \cite{SR} stated that, if $$R(q)=\frac{(q^2; q^5)_\infty(q^3;q^5)_\infty}{(q;q^5)_\infty(q^4;q^5)_\infty}$$ then
\begin{equation}\label{eq2}
\left(q;q\right)_\infty=\left(q^{25};q^{25}\right)_\infty\left(R(q^5)-q-\frac{q^2}{R(q^5)}\right).
\end{equation}
One can rewrite \eqref{eq2} as
\begin{equation}\label{eq3}
\eta=q^{-1}R-1-qR^{-1}, 
\end{equation}where \begin{equation}\label{eq1a}
\eta=\frac{\left(q;q\right)_\infty}{q\left(q^{25};q^{25}\right)_\infty}.
\end{equation}
Using \eqref{eq3} Hirschhorn and Hunt \cite{MD}  proved the following lemma:
\begin{lemma}\cite{MD}\label{flem}
We have$${H_5(\eta)}=-1,H_5{(\eta^2)}=-1,H_5(\eta^3)=5,  H_5(\eta^4)=-5.$$  where $H_5$ is an operator  which acts on series of positive and negative powers of a single variable, and simply picks out the term in which the power is congruent to 0 modulo 5.
\end{lemma}
\noindent Note: Here we replaced the symbol $H$ in \cite{MD} by $H_5$.

\begin{lemma}\label{lem2} For any prime $p$, we have
$$\left(q;q\right)_\infty^5\equiv\left(q^5;q^5\right)_\infty\pmod 5.$$ 
\end{lemma}
\begin{proof}This follows easily from the binomial theorem.\end{proof}

\begin{lemma}\label{lem4}\cite[p. 53]{bcb}, We have
$$(q;q)^3_\infty=\sum_{k=0}^\infty(-1)^k(2k+1)q^{k(k+1)/2}.$$
\end{lemma}

\section{Proof of Theorems \ref{thm0}-\ref{thm6}}
In this section, all congruences are to the modulus 5.

\noindent\textbf{\textit{{Proof of Theorem \ref{thm0}}:}} 
 Setting $r=-(5\lambda+1)$ in  \eqref{eq1}, we obtain
\begin{equation}\label{eq3.4jj}
\sum_{n=0}^\infty{p_{-(5\lambda+1)}(n)q^n}=(q;q)_\infty^{5\lambda+1}=(q;q)_\infty^{5\lambda}(q;q)_\infty.
\end{equation}
Employing Lemma \ref{lem2} and \eqref{eq2} in \eqref{eq3.4jj}, we obtain
\begin{equation}\label{eq3.5jj}
\sum_{n=0}^\infty{p_{-(5\lambda+1)}(n)q^n}\equiv(q^5;q^5)_\infty^{\lambda}\left(q^{25};q^{25}\right)_\infty\left(R(q^5)-q-\frac{q^2}{R(q^5)}\right).
\end{equation}
Extracting the terms involving $q^{5n+\ell}$, where $\ell=3,4$, in both sides of \eqref{eq3.5jj}, we arrive at the desired result.
\vskip 3mm

\noindent\textbf{\textit{{Proof of Theorem \ref{thm1}}:}} 
  Setting $r=-(5\lambda+3)$ in  \eqref{eq1}, we obtain
\begin{equation}\label{eq1.4}
\sum_{n=0}^\infty{p_{-(5\lambda+3)}(n)q^n}=(q;q)_\infty^{5\lambda+3}=(q;q)_\infty^{5\lambda}(q;q)^3_\infty.
\end{equation}
Employing Lemma \ref{lem2} in \eqref{eq1.4}, we obtain
\begin{equation}\label{eq1.5}
\sum_{n=0}^\infty{p_{-(5\lambda+3)}(n)q^n}\equiv(q^5;q^5)_\infty^{\lambda}(q;q)^3_\infty.
\end{equation}
Using Lemma \ref{lem4} in \eqref{eq1.5}, we obtain
\begin{equation}\label{eq1.6}
\sum_{n=0}^\infty{p_{-(5\lambda+3)}(n)q^n}\equiv(q^5;q^5)_\infty^{\lambda}\times\sum_{k=0}^\infty(-1)^k(2k+1)q^\frac{k(k+1)}{2}.
\end{equation} Since there exist no non-negative integer $k$ such that  $k(k+1)/2$ is congruent to 2 or 4 modulo 5, so extracting the terms involving $q^{5n+k}$, where $k=2, 4$, in both sides of \eqref{eq1.6}, we prove the cases $\ell=2$ and 4.

Again,  extracting the terms involving $q^{5n+3}$ in both sides of \eqref{eq1.6} and  noting $2k+1\equiv 0 \pmod 5$ for $k=2$, we arrive at the case $\ell=3$.

\vskip 3mm
\noindent\textbf{\textit{{Proof of Theorem \ref{thm3}}:}} 
 Setting $r=-(5\lambda+4)$ in  \eqref{eq1}, we obtain
\begin{equation}\label{eq3.4}
\sum_{n=0}^\infty{p_{-(5\lambda+4)}(n)q^n}=(q;q)_\infty^{5\lambda+4}=(q;q)_\infty^{5\lambda}(q;q)^4_\infty.
\end{equation}
Employing Lemma \ref{lem2} and \eqref{eq1a} in \eqref{eq3.4}, we obtain
\begin{equation}\label{eq3.5}
\sum_{n=0}^\infty{p_{-(5\lambda+4)}(n)q^n}\equiv(q^5;q^5)_\infty^{\lambda}{\eta}^4q^4(q^{25};q^{25})^4_\infty.
\end{equation}
Extracting the terms involving $q^{5n+4}$ and using the operator $H_5$ in \eqref{eq3.5}, we obtain
\begin{equation}\label{eq3.6}
\sum_{n=0}^\infty{p_{-(5\lambda+4)}(5n+4)q^{5n+4}}\equiv(q^5;q^5)_\infty^{\lambda}H_5\left({\eta}^4\right)q^4(q^{25};q^{25})^4_\infty.
\end{equation}
Using Lemma \ref{flem} in \eqref{eq3.6}, we obtain
\begin{equation}\label{eq3.7}
\sum_{n=0}^\infty{p_{-(5\lambda+4)}(5n+4)q^{5n+4}}\equiv(-5)(q^5;q^5)_\infty^{\lambda}q^4(q^{25};q^{25})^4_\infty.
\end{equation}
Now the desired result follows from  \eqref{eq3.7} and the fact that $5\equiv0\pmod5$.

\vskip 3mm

\noindent\textbf{\textit{{Proof of Theorem \ref{thm5}}:}} 
Setting $r=-(25\lambda+1)$ in  \eqref{eq1}, we obtain
\begin{equation}\label{eq3.4j}
\sum_{n=0}^\infty{p_{-(25\lambda+1)}(n)q^n}=(q;q)_\infty^{25\lambda+1}=(q;q)_\infty^{25\lambda}(q;q)_\infty.
\end{equation}
Employing Lemma \ref{lem2} and \eqref{eq1a} in \eqref{eq3.4j}, we obtain
\begin{equation}\label{eq3.5j}
\sum_{n=0}^\infty{p_{-(25\lambda+1)}(n)q^n}\equiv(q^{25};q^{25})_\infty^{\lambda}{\eta}q(q^{25};q^{25})_\infty.
\end{equation}
Extracting the terms involving $q^{5n+1}$ and using the operator $H_5$ in \eqref{eq3.5j}, we obtain
\begin{equation}\label{eq3.6j}
\sum_{n=0}^\infty{p_{-(25\lambda+1)}(5n+1)q^{5n+1}}\equiv{q(q^{25};q^{25})_\infty^{\lambda+1}}H_5\left({\eta}\right).
\end{equation}
Using Lemma \ref{flem} in \eqref{eq3.6j},  dividing by $q$, and  replacing $q^5$ by $q$ , we obtain
\begin{equation}\label{eq3.7j}
\sum_{n=0}^\infty{p_{-(25\lambda+1)}(5n+1)q^n}\equiv(-1)(q^5;q^5)_\infty^{\lambda+1}.
\end{equation}
Extracting the terms involving $q^{5n+\ell}$, where $\ell=1,2,3,4,$ in both sides of \eqref{eq3.7j}, we arrive at the desired result. 
\vskip 3mm

\noindent\textbf{\textit{{Proof of Theorem \ref{thm6}}:}}
Setting $r=-(25\lambda+2)$ in  \eqref{eq1}, we obtain
\begin{equation}\label{eq3.4jjj}
\sum_{n=0}^\infty{p_{-(25\lambda+2)}(n)q^n}=(q;q)_\infty^{25\lambda+2}=(q;q)_\infty^{25\lambda}(q;q)^2_\infty.
\end{equation}
Employing Lemma \ref{lem2} and \eqref{eq1a} in \eqref{eq3.4jjj}, we obtain
\begin{equation}\label{eq3.5jjj}
\sum_{n=0}^\infty{p_{-(25\lambda+2)}(n)q^n}\equiv(q^{25};q^{25})_\infty^{\lambda}{\eta}^2q^2(q^{25};q^{25})^2_\infty.
\end{equation}
Extracting the terms involving $q^{5n+2}$ and using the operator $H_5$ in \eqref{eq3.5jjj}, we obtain
\begin{equation}\label{eq3.6jj}
\sum_{n=0}^\infty{p_{-(25\lambda+2)}(5n+2)q^{5n+2}}\equiv{q^2(q^{25};q^{25})_\infty^{\lambda+2}}H_5\left({\eta}^2\right).
\end{equation}
Using Lemma \ref{flem} in \eqref{eq3.6jj}, dividing by $q^2$, and replacing $q^5$ by $q$ , we obtain
\begin{equation}\label{eq3.7jj}
\sum_{n=0}^\infty{p_{-(25\lambda+2)}(5n+2)q^n}\equiv1\times(q^5;q^5)_\infty^{\lambda+2}.
\end{equation}
Extracting the terms involving $q^{5n+\ell}$, where $\ell=1,2,3,4,$ in both sides of \eqref{eq3.7jj}, we arrive at the desired result.

\end{document}